\documentclass[12pt,reqno]{amsart}
\usepackage{amssymb,amsmath}
\usepackage{algorithmic,algorithm}
\usepackage{amsfonts}
\usepackage{graphicx}
\usepackage{pstricks}
\usepackage{verbatim}
\theoremstyle{plain}
\numberwithin{equation}{section}
\newtheorem{thm}{Theorem}[section]
\newtheorem{theorem}[thm]{Theorem}
\newtheorem{lemma}[thm]{Lemma}

\newtheorem{definition}[thm]{Definition}

\newtheorem{corollary}[thm]{Corollary}
\newtheorem{remark}[thm]{Remark}

\calclayout
\begin{document}

\title[]{A Problem on the Largest Divisor $d$ \\
of $N$ with $d\leq \sqrt{N}$}
\author{Srikanth Cherukupally}
\thanks{Email Address: sricheru1214@gmail.com}
\email{sricheru1214@gmail.com}
\maketitle
\begin{abstract}
 For a given number $N$, we consider the problem of
 computing two integers $1\leq r,f < N$ such that
 the set
 \begin{displaymath}
 \mathcal{X}(N,r,f) = \{(a+b)-(f+\frac{Nr+1}{f}): ab=Nr\}
 \end{displaymath}
 consists only of positive integers. Computing a solution to the problem is equivalent to finding a
 pair $(r,f)$ satisfying $l(Nr) < f \leq l(Nr+1)$, where $l(x)$ is the largest divisor of $x$ bounded by $\sqrt{x}$. This requires factoring both $Nr$ and $Nr+1$. We present a simple randomized algorithm that - avoiding factoring - computes pairs $(r,f)$. We give an exact formula for the total number of possible pairs $(r,f)$, and with the aid of empirical data we estimate that the ratio
 $$\frac{\phi(N)-2}{|\mathfrak{F}(N)|}$$ is approximately about $c*\log \log N$. Here, $\mathfrak{F}(N)$ is the set of unique $r$ appearing among all possible pairs $(r,f)$, $\phi(.)$ is the Euler's totient function, and $c$ is a constant equal to 2 for prime $N$ and oscillates much for composite $N$.

 As a separate and independent case, we study the same problem of computing
 $(r,f)$ with $r>N$. We present a procedure to find such an $r$, which requires finding the least prime in an arithmetic progression.

 \end{abstract}
\section{Introduction}
\noindent We first note that the set $\mathcal{X}(N,r,f)$ contains integers only when $f$
is a divisor of $Nr+1$. Thus, \textit{whenever $\mathcal{X}(N,r,f)$  
is used, it is implicit that $f$ is a divisor of $Nr+1$.} \\

We use the following notation throughout the paper.  
\begin{itemize}
\item To denote that the set $\mathcal{X}(N,r,f)$ consists of only positive integers,
we use $[\mathcal{X}(N,r,f)]>0$, otherwise we use $[\mathcal{X}(N,r,f)]\not >0$, in which case the set contains
negative integers and may contain zero.
\item For integer $n\geq 1$, $l(n)$ denotes the 
largest divisor $d$ of $n$, with $d\leq \sqrt{n}$. \\ 
For example, $l(3^2.5)=5$, $l(2^2.3.5) = 6$. \\
For $n$ prime, $l(n)=1$. If $n$ is a perfect square, $l(n)=\sqrt{n}$.
\end{itemize}

We consider both combinatorial and algorithmic aspects concerning the set $\mathcal{X}(N,r,f)$.
The main results of the paper are presented below.

\subsection*{Algorithmic aspect.} We prove the fact that computing a solution to the defined problem, i.e., a pair $(r,f)$ requires finding a random $r$ that satisfies $l(Nr) < l(Nr+1)$. Finding such an $r$ requires factoring both $Nr$ and $Nr+1$. We present a randomized algorithm that - avoiding factoring - chooses a random $1\leq f < N$ and computes $r$ to find a pair $(r,f)$. For a prime $N$, the generated $(r,f)$ is always such that $[\mathcal{X}(N,r,f)]>0$. For a composite $N$, we prove the success rate of the
procedure, meaning for a generated pair $(r,f)$ we have $[\mathcal{X}(N,r,f)]>0$ with probability $> 0.75$.
The algorithm involves one modular inverse computation $\pmod N$, so
the running time is $O(\log^3 N)$.

As a separate and independent case, we study the same problem of computing
 $(r,f)$ with $r>N$. A brute-force method exists to find such an $r$, but requires finding the least prime in an arithmetic progression. From \cite{heath1, linnik1, linnik2}, the least prime in arithmetic progression
 $\pmod N$ is $O(N^{5.2})$. The brute-force method is not efficient, its complexity is exponential in $\log N$.

\subsection*{Combinatorial aspect.} For a given integer $N$, let $G(N)$ be the number of
distinct pairs $(r,f)$ such that $[\mathcal{X}(N,r,f)]>0$. We prove that
\begin{displaymath}
 |G(N)| \leq \phi(N) - 2*A(2,l(N)) - 2^u.
\end{displaymath}
Here, $u$ is the number of distinct prime factors of $N$, the function
$A(2,l(N))$ gives the number of integers in the interval $[2,l(N)]$ that are co-prime to $N$.
When $N$ is prime, $A(2,l(N))$ is zero, $2^u-2$ is zero, and thus $|G(N)| = \phi(N) - 2$.

Let $\mathfrak{F}(N)$  = $\{r: (r,f)\in G(N)\}$. From the empirical data, we have, for prime $N$,
\begin{displaymath}
  |\mathfrak{F}(N)| = \frac{\phi(N)-2}{2(\log\log(N) - C_N)}
\end{displaymath}
It is observed that $C_N$ increases with increasing $N$. Proving rigorously the observed behaviour of
$|\mathfrak{F}(N)|$ is left as an {\bf open issue}.

\section{An equivalent Condition for $[\mathcal{X}(N,r,f)]>0$}

For integer $k\geq 1$, let $D(k)$ be the set of divisors of $k$. 
Define a function $f: D(k) \rightarrow \mathbb{Q}$ as $f(x) = x+\frac{k}{x}$.  
Let $min_f(k) = min\{ f(x): x\in D(k)\}$.
\begin{lemma}
For $k\geq 1$, $min_f(k) = l(k) + \frac{k}{l(k)}$. 
\end{lemma}
\begin{lemma} 
 If one of $Nr$ and $Nr+1$ is a perfect square, then for any divisor $f$ of $Nr+1$,
 we have $[\mathcal{X}(N,r,f)]\not >0$
\end{lemma}
\proof{Clearly, only one of $Nr$ and $Nr+1$ can be a perfect square. We prove the result for both
cases. The arguments are similar in both cases. \\

Suppose $Nr$ is 
a perfect square. Then, for any divisor $f$ of $Nr+1$, 
$\mathcal{X}(N,r,f)$ consists of the element $t = 2\sqrt{Nr} - (f+\frac{Nr+1}{f})$. Since 
$f<\sqrt{Nr}$,  $f+\frac{Nr+1}{f}>2\sqrt{Nr}$. Thus $t<0$. \\ 

Suppose $Nr+1 = z^2$ for some integer $z>0$. Then, for any divisor $f$ of $Nr+1$, 
$\mathcal{X}(N,r,f)$ consists of the element $s = 2z - (f+\frac{Nr+1}{f})$. Since 
$f\leq z$,  $f+\frac{Nr+1}{f}\geq 2z$. Thus, $s\leq 0$. } \hfill $\Box$ 
\begin{theorem}
\label{main-thm}
Suppose $N$ and $r$ are two integers such that neither of $Nr$ and $Nr+1$ is a perfect square. Let 
$f$ be a divisor of $Nr+1$ with $1\leq f<\sqrt{Nr+1}$. Then, 
\begin{center}
$[\mathcal{X}(N,r,f)]>0$ if and only if $l(Nr)<f \leq l(Nr+1)$.  
\end{center}
\end{theorem} 
\proof{For integers $A,B$ with $AB = Nr$, and $f$ a divisor of $Nr+1$, let 
$T(A,B,f) = A+B-(f+\frac{Nr+1}{f})$. The set $\mathcal{X}(N,r,f)$ consists of 
all elements $T(A,B,f)$ where $AB=Nr$. From Lemma 1, the least element of 
$\mathcal{X}(N,r,f)$ is $T(l(Nr), Nr/l(Nr),f)$. 
We show that 
 $T(l(Nr), Nr/l(Nr),f)>0$ if and only if $l(Nr)<f\leq l(Nr+1)$. This result is equivalent to 
 proving the claim.  Consider the product
\begin{eqnarray}
f\times T(l(Nr), Nr/l(Nr),f) & = & f \big[l(Nr) + \frac{Nr}{l(Nr)} - (f+\frac{Nr+1}{f})\big] \nonumber \\
               & = & (f-l(Nr))\bigg[\frac{Nr}{l(Nr)}-f\bigg]-1 \nonumber 
\end{eqnarray} 
Since both $Nr$, $Nr+1$ are not perfect squares, $f, l(Nr)<\sqrt{Nr}$. Thus, $\frac{Nr}{l(Nr)}-f>1$. 
From the above equality, if $f>l(Nr)$, then 
$T(l(Nr), Nr/l(Nr),f)>0$, and also $T(l(Nr), Nr/l(Nr),f)>0$ only when $f>l(Nr)$. This completes the 
argument. } \hfill $\Box$ \\

From Theorem \ref{main-thm}, we have the following corollary.
\begin{corollary}
 Computing a pair $(r,f)$ such that $[\mathcal{X}(N,r,f)]>0$
equivalent to computing integers $f,r$ such that $f$ is a divisor of $Nr+1$ with $l(Nr)<f\leq l(Nr+1)$.
\end{corollary}

For a randomly chosen $1\leq r<N$,  computing  $l(Nr+1)$, $l(Nr)$
and verifying whether there exists a divisor $f$ of $Nr+1$ such that 
$l(Nr)<f\leq l(Nr+1)$ require factoring both $Nr$ and $Nr+1$. An interesting
question is: \\

{\it Can we compute a pair $(r,f)$ without having to factor $Nr$, $Nr+1$,
more importantly without having to factor $N$?} \\

In the following section we give a randomized algorithm to compute a pair $(r,f)$. For prime $N$, the computed pair is always a solution. But, for composite $N$, verifying whether the computed pair is a solution requires checking if $r$ is prime or not.

\section{Procedure to compute a pair $(r,f)$}
\begin{definition}
A pair $(r,f)$ is called a valid pair if $[\mathcal{X}(N,r,f)]>0$.
\end{definition} 

A procedure for finding pairs $(r,f)$ is given as Algorithm \ref{alg}.
 We prove the following for a generated pair.
\begin{itemize}
 \item when $N$ is prime, a pair generated by the algorithm is certainly valid
\item  when $N$ is composite, a pair generated by the algorithm is
valid with probability $ > 3/4$. So, in this case the algorithm is probabilistic.
\end{itemize}

 \begin{algorithm}[H]
 \caption{Algorithm for computing $(r,f)$}
 \label{alg}
 Input:  An integer $N$, a real number $1/2\leq k<1$  \\
 Output: A pair $(r,f)$
\label{alg:power}
\begin{algorithmic}[1] 
 \IF{$N$ is prime} 
 \STATE Choose $f\in (1,N-1)$ 
 \STATE Compute $f^{'}$ such that $ff^{'}=1\pmod{N}$ 
 \STATE Compute $r = (ff^{'}-1)/N$ 
 \STATE Output $(r,f)$
 \ENDIF 
 \IF{$N$ is composite}
 \STATE Choose $f \in (kN,N-1)$ co-prime to $N$
 \STATE Compute $f^{'}$ such that $ff^{'}=1\pmod{N}$ 
 \STATE Compute $r = (ff^{'}-1)/N$ 
 \STATE \textbf{if} $r<kf$ repeat steps 8-10 
 \STATE \textbf{else} output $(r,f)$.
 \ENDIF  
 \end{algorithmic}
\end{algorithm}  

For a given integer $N$, let $(r,f)$ be a pair generated by the algorithm. We know that
the set $\mathcal{X}(N,r,f)$ consists of integers $A+B - (f + \frac{Nr+1}{f})$  where $AB=Nr$.
There are four possible ways in which $N$ and $r$ can split between
$A$ and $B$ in the expression $A+B - (f + \frac{Nr+1}{f})$, based on the compositeness of $N$ and $r$.
We have 4 different cases: Case 1, Case 2, Case 3,
and Case 4. The elements of $\mathcal{X}(N,r,f)$ fall into one of these cases. 
For each case except Case 4, we find that
the elements of $\mathcal{X}(N,r,f)$ are greater than
$0$. But, for Case 4, we need a probability argument for the elements to be greater than $0$.

\subsection*{Case 1: (both $N$ and $r$ are prime) }
 \begin{displaymath}
  A+B = N+r
 \end{displaymath}
 Since $r<f$, $A+B - (f + \frac{Nr+1}{f})>0$. 
\subsection*{Case 2: ($N$ is composite, and $r$ is prime)} Let $N = n_1n_2$. Without loss of
 generality, $1<n_1\leq n_2$. We have $n_2 \leq \frac{N}{2}$ as $n_1\geq 2$, and thus $A+B$ can be expressed in two possible ways:
 \begin{displaymath}
   A+B = n_1r + n_2 \,\textrm{(or)}\, n_1 + n_2r
 \end{displaymath}
  When $N$ is composite, $f$ is chosen to be greater than 
 $\frac{N}{2}$.  Thus, we have $n_1,n_2<f$. Therefore, in both sub cases
 $A+B - (f + \frac{Nr+1}{f})>0$. 
 \subsection*{Case 3: ($N$ is prime, and $r$ is composite) }
 Suppose $r = r_1r_2$. Without loss of generality, suppose that $1<r_1\leq r_2$. Then,
 $r$ can split between $A$ and $B$
 in two possible ways:
  \begin{displaymath}
   A+B = nr_1 + r_2 \,\textrm{(or)}\, r_1 + nr_2
 \end{displaymath}
Since $r_1,r_2<r<f$,  $A+B - (f + \frac{Nr+1}{f})>0$. 
\subsection*{Case 4: (both $N$ and $r$ are composite) }
Let $N=n_1n_2$, and $r=r_1r_2$.
 \begin{displaymath}
   A+B = n_1r_1 + n_2r_2 
 \end{displaymath} 
 where $n_1r_1<\sqrt{Nr}$. As $f$ satisfies 
 $\sqrt{kNr}<f<\sqrt{Nr}$. 
 If $n_1r_1<f$, then $A+B - (f + \frac{Nr+1}{f})>0$. 
 
In the algorithm, $f^{'}$ is the multiplicative inverse of $f \pmod N$.
So, $ff^{'} = 1+Nr$ for some $1\leq r<f$.  With probability $1-k$,  $r$ 
falls in $[kf,f)$. So, we have the conditions: $kN < f<N$ and $kf < r<f$.  
From these conditions, we can deduce $\sqrt{k}f< \sqrt{Nr}< f/\sqrt{k}$,
equivalently, $\sqrt{kNr}< f < \sqrt{Nr/k}$.

\subsection*{When $N$ is prime}
 Only Case 1 and Case 3 happen.
 In these two cases, for any chosen $1<f<N-2$, and the computed $r$ by the algorithm, $[\mathcal{X}(N,r,f)] > 0$. Thus, a pair generated by the algorithm is always a {\bf valid pair}.
 \subsection*{When $N$ is composite}
  We get two separate conditions depending on whether $r$ is prime or not.
\subsubsection*{When $r$ is prime} We are in Case 2, in which the corresponding elements of $\mathcal{X}(N,r,f)$ are greater than $0$. Thus, $[\mathcal{X}(N,r,f)]>0$. Thus the computed pair $(r,f)$ is a {\bf valid pair}.

\subsubsection*{When $r$ is composite} We are in Case 4, in which
 we may not be able to tell with certainty that $[\mathcal{X}(N,r,f)]>0$.
 It is known from Case 4 that $\sqrt{kNr}<f<\sqrt{Nr}$. Thus, 
 if $l(Nr)<\sqrt{kNr}$, then $[\mathcal{X}(N,r,f)]>0$.   
 We prove the following result.
 \begin{theorem}
 \label{prob}
 The probability that $l(Nr)<\sqrt{kNr}$ is at least $\frac{1}{2}+\frac{k}{2}$. 
 \end{theorem}
  
  As a consequence of the above result, we obtain the success rate of the algorithm for composite $N$. \\
  
\begin{remark}
 \textit{For a composite $N$, a pair $(r,f)$ generated by the algorithm is valid if $r$ is prime,
  and is valid with probability $> 3/4$ if $r$ is composite. }
\end{remark}
\section{Proof of Theorem \ref{prob}}
 \noindent For integer $x$, a real number $0<k<1$, define the set
 \begin{displaymath}
\tau(x,k) = \{r\leq x: l(r) < \sqrt{kr} \}  .
 \end{displaymath}
 In establishing a lower bound on the size of $\tau$, we work with its complement set $\bar{\tau}$. 
 The set $\bar{\tau}(x,k)$ consists of 
 integers $r\leq x$ such that $\sqrt{kr}\leq l(r)\leq \sqrt{r}$. \\ 
 
 For integer $a\geq 1$, define  $S(x;a,k) = \{ab\leq x: a<b<a/k\}$. Clearly, $a/k<x/a$. This inequality 
 implies that $a\leq \sqrt{xk}$. Thus, $|S(x;a,k)| = (\frac{1}{k}-1)a$.  
 \begin{lemma}
 For any $0<k\leq 1$
 \begin{displaymath}
  \bar{\tau}(x,k) = \bigcup_{a=1}^{\sqrt{xk}} S(x;a,k)
  \end{displaymath}
  \end{lemma}
\proof{For $ab\in S(x;a,k)$,  $\sqrt{kab}\leq a \leq \sqrt{ab}$, which satisfies the 
definition of $\bar{\tau}$. Thus $ab\in \bar{\tau}$. Conversely, suppose 
$r\in \bar{\tau}(x,k)$. The integer $r = ab$ where $a = l(r)$, $b = \frac{r}{l(r)}$. 
Clearly, $r\in S(x;l(r),k)$ since the condition $\sqrt{kr}\leq l(r)\leq \sqrt{r}$
 implies that $l(r)< \frac{r}{l(r)}< l(r)/k$. This proves the claim. }  
 \hfill $\Box$   \\
 
 From the above result, we have  
 \begin{eqnarray}
  |\bar{\tau}(x,k)|  \leq  \sum _{a=1}^{\lfloor \sqrt{xk} \rfloor } |S(x;a,k)|  
                    & = & (\frac{1}{k}-1) \sum_{a=1}^{\lfloor \sqrt{xk} \rfloor}  a  \nonumber \\
                    & \leq &  (\frac{1}{k}-1)\bigg( \frac{xk}{2}+ \frac{1}{2}\sqrt{xk} \bigg). \nonumber 
 \end{eqnarray}
So, $|\tau(x,k)| = x -  |\bar{\tau}(x,k)| \geq (\frac{1}{2}+\frac{k}{2})x - \frac{\sqrt{xk}}{2}$.
Thus, the probability that $n\in \tau(x,k)$ is $\frac{|\tau(x,k)|}{x} 
\geq \frac{1}{2}+\frac{k}{2} - \frac{\sqrt{k}}{2\sqrt{x}}$. The asymptotic probability is $1/2+k/2$. This 
completes the proof. 
\section{Combinatorial Aspect}
\noindent For a given integer $N\geq 1$, we want to estimate the number of distinct
$r$ such that $[\mathcal{X}(N,r,f)]>0$. We discuss this problem in two 
separate cases: when $N$ is prime and when $N$ is composite.  \\

For integer $N\geq 1$, define
\begin{displaymath}
 \mathfrak{F}(N) = \{ r : \exists f, f|Nr+1, [\mathcal{X}(N,r,f)]>0 \}
\end{displaymath}
By Theorem \ref{main-thm}, we know that $\mathfrak{F}(N)$ is the set of 
integers $r$ for which there is a divisor $f$ of $Nr+1$ with the property: $l(Nr)<f\leq l(Nr+1)$. \\

To estimate the size of $\mathfrak{F}(N)$,  we first calculate the number of distinct pairs $(r,f)$
such that $[\mathcal{X}(N,r,f)]>0$ where $1\leq r<N$, and $1\leq f<N$. For each such pair $(r,f)$, $r < f$.  Further, we note the following two properties.
\begin{enumerate}
 \item The same $r$ can appear in more than one $(r,f)$ pairs. It is because $Nr+1$ may have more than one
 divisor which fall in the interval $(l(Nr), l(Nr+1))$.
 \item An $f$ can appear only in at most one $(r,f)$ pair. The proof of this fact is by contradiction.
 Suppose there are two pairs $(f,r_1)$, $(f,r_2)$ with $r_1\not = r_2$. This implies that 
 $f$ divides both $Nr_1+1$, $Nr_2+1$, and thus $f|N(r_1-r_2)$. Since $r_1,r_2<f$, $r_1=r_2$, a contradiction. 
\end{enumerate} 

\subsection{Binary matrix $\mathcal{M}(N)$}
For integer $N$, Let $\mathcal{M}(N)$ be a binary matrix defined as 
follows. 
\begin{displaymath}
 \mathcal{M}(N) = (m_{f,r}),  1\leq f<N, 1 \leq r<N
\end{displaymath}
where 
\begin{displaymath}
 m_{fr} = \left\{ \begin{array}{ll}
                   1 & \textrm{if $gcd(f,N)=1$ and $[\mathcal{X}(N,r,f)]>0$} \\
                   0 & \textrm{if $gcd(f,N)=1$ and $[\mathcal{X}(N,r,f)]\not >0$} \\
                   0 & \textrm{if $gcd(f,N)>1$} 
                  \end{array} \right.
\end{displaymath}

The matrix takes the following structure (see Figure \ref{fig:my_image}).
\begin{itemize}
 \item Each row of the matrix gets at most $1$ non-zero entry due to the second property of $(r,f)$ pairs.
 \item All non-zero entries appear in below diagonal region of the matrix since $(r,f)$ is a pair with $r<f$. So, the matrix is lower triangular.
 \item  The first row of $\mathcal{M}(N)$ consists of all zeros since $[\mathcal{X}(N,r,1)]\not >0$. 
\end{itemize}

\begin{figure}[htbp]
    \centering 
    \includegraphics[width=0.7\linewidth]{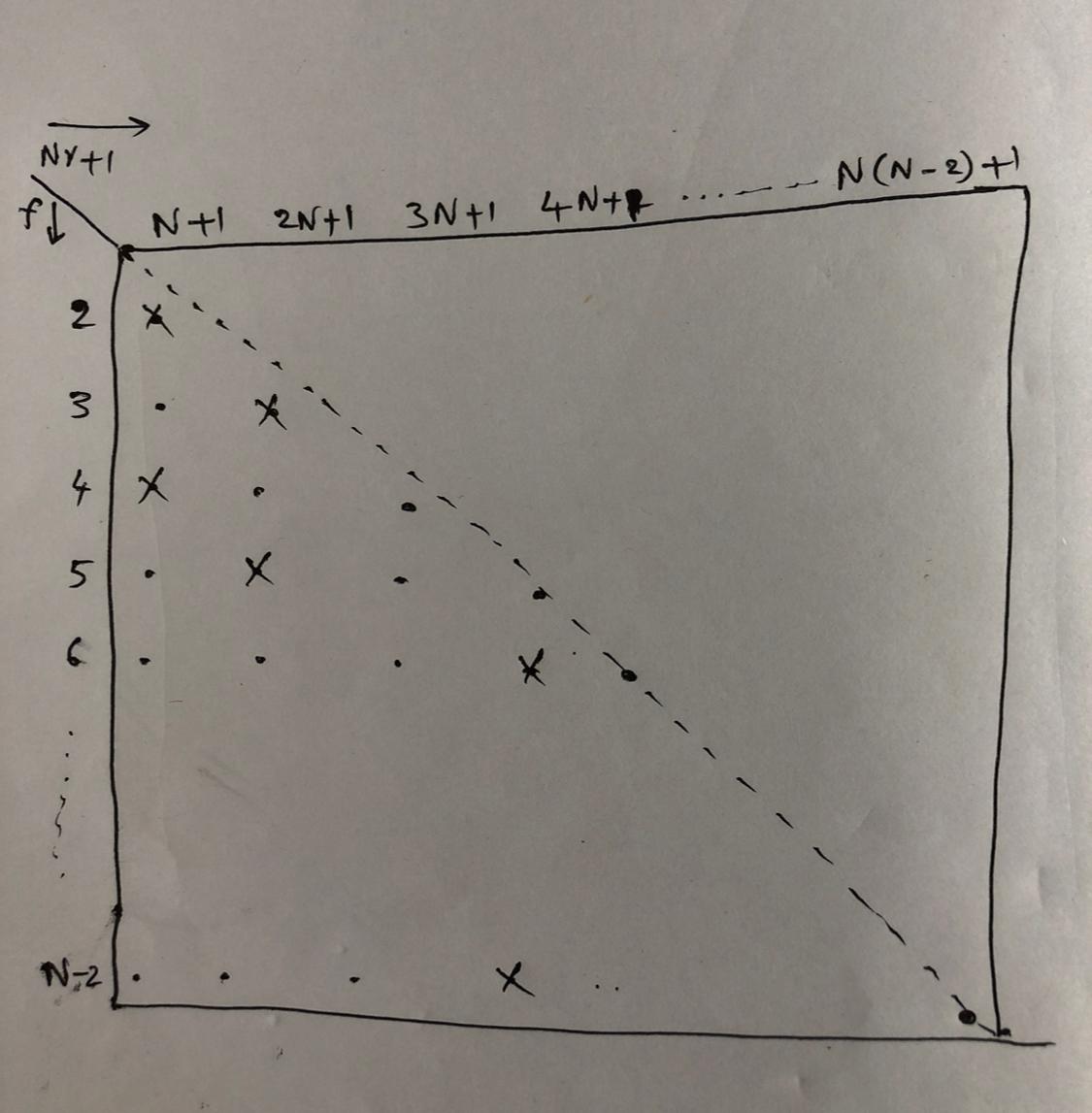} 
    \caption{The matrix depicts for a prime $N$. A row corresponding to an $f$ dividing $Nr+1$ that represents a column. Each row has only one 'X' denoting that $f$ divides one $Nr+1$ with $r < f$. The symbols 'X' spread across columns only in below diagonal region. } 
    \label{fig:my_image} 
\end{figure}

\noindent \textbf{Definition.} A row of $\mathcal{M}$ is called an \textit{active} row if it gets a non-zero entry.    \\

Let $k_r$ represent the number of non-zero entries in $r^{th}$ column of the matrix. The quantity $k_r$ is nothing but the number of divisors $f$ of $Nr+1$ that are $> L(Nr)$. We have
\begin{displaymath}
\sum_{r=1}^{N} k_r = \textrm {the number of active rows} 
\end{displaymath}  

\begin{lemma} 
For $1\leq r<N$, $k_r$ is an even integer $\geq 0$ provided $Nr+1$ is not 
a perfect square.  
\end{lemma}
\proof{The quantity $k_r$ is the number of divisors $f$ of $Nr+1$ with 
$l(Nr)<f\leq l(Nr+1)$. If $f$ is a divisor of $Nr+1$ with 
$l(Nr)<f\leq l(Nr+1)$, then $[\mathcal{X}(N,r,f)]>0$, and also
$[\mathcal{X}(N,r,\frac{Nr+1}{f})]>0$.  } \hfill $\Box$  
\subsection{The size of $\mathfrak{F}(N)$ when $N$ is prime}

\begin{lemma}
 For prime $N$, the number of active rows in $\mathcal{M}(N)$ is $\phi(N)-2$
\end{lemma}
\proof{For a chosen $2\leq f<N$, we have $1\leq r<f$ 
such that $f$ divides $1+Nr$. Since $N$ is prime, $f>l(Nr)=r$, 
and thus $\mathcal{X}(N,r,f)$ consists of all positive integers. Therefore 
$r\in \mathfrak{F}(N)$. This argument is true for any $f>1$ that is co-prime to $N$. When 
$f = N-1$, $\mathcal{X}(N,r,f)$ consists of all positive integers except one zero. 
This implies that the number of $(r,f)$ pairs s.t. $\mathcal{X}(N,r,f)$ has all positive values
is exactly equal to $\phi(N)-2$. } \hfill $\Box$ \\

From the above result, we have 
\begin{displaymath}
 \sum_{r=1}^{N-3} k_r = \phi(N)-2
\end{displaymath} 
 
Since $N$ is prime, $rN+1$ is not a perfect square for any $1\leq r\leq N-3$, and thus $k_r$ is even.  \\

Suppose each column in $\mathcal{M}(N)$ gets not more than two entries, i.e, 
for $1\leq i\leq N-3$, $k_i=2$ or $0$. Then, $|\mathfrak{F}(N)|$ will be equal to 
$\frac{\phi(N)-2}{2}$. This gives us the upper bound for $|F(n)|$. 
\begin{displaymath}
|\mathfrak{F}(N)| \leq \frac{\phi(N)-2}{2} = \frac{N-3}{2}. 
\end{displaymath}
The upper bound is attained for some primes. For 
example: for $N=5,7$, $|\mathfrak{F}(N)|= \frac{\phi(N)-2}{2}$. But, for $N=11$, $|\mathfrak{F}(N)| = \frac{\phi(N)-2}{2}-1$. 
Empirically, it is observed that only $5$ and $7$ are only primes
for which $|\mathfrak{F}(N)|$ attains the maximum values.  As $N$ increases, 
$|\mathfrak{F}(N)|$ tends to get much smaller than $\frac{\phi(N)-2}{2}$. This behaviour is in accordance 
with our intuition that as $N$ increases, the number of entries in a given column of the matrix increases.
It is interesting to ask the following question.
\begin{center}
How small the size of $\mathfrak{F}(N)$ could get? 
\end{center}

Considering the average behaviour of divisor function, one would get a lower bound $\frac{\phi(N)-2}{2*\log(N)}$. However, this is a not-so-tight lower bound as we are looking for divisors $f$ that are subject to the condition that $f > L(nr)$. To confirm, we have gathered some empirical evidence.

\subsubsection*{Empirical observation}
We compare the ratio $\frac{\phi(N)-2}{2|\mathfrak{F}(N)|}$ and the quantity $\log \log(N)$. For initial values of
$N$, $\frac{\phi(N)-2}{2|\mathfrak{F}(N)|}> \log\log(N)$. For the first time, when $N = 1151$, the ratio
becomes less than $\log\log N$. After that until $N = 25453$, one or the other is bigger.
After that point, for every prime $N>25453$, the ratio is bigger than $\log\log N$. To account for
 observed difference between the two quantities we introduce a variable $C_N$.

\begin{displaymath}
  \frac{\phi(N)-2}{2|\mathfrak{F}(N)|} = \log\log(N) - C_N
\end{displaymath}
Equivalently, 
\begin{displaymath}
  |\mathfrak{F}(N)| = \frac{\phi(N)-2}{2(\log\log(N) - C_N)}
\end{displaymath} 
The data in Table 1 shows that $C_N$ increases with increasing $N$.
\begin{table}
\caption{data on $\mathfrak{F}(N)$}
\begin{center}
\begin{tabular}{|l|l|l|l|}
\hline
The magnitude of $N$ & $\frac{\phi(N)-2}{2|\mathfrak{F}(N)|}$ & $\log\log N$ & $C_N$  \\
\hline
$10^5$ & 2.37  & 2.4434 & 0.05 \\
\hline
$10^6$ & 2.510 & 2.625 & .109 \\
\hline
$10^7$ & 2.6275 & 2.7799 &  0.15 \\
\hline
$10^8$ & 2.7265 & 2.9134 &  0.2069 \\
\hline
\end{tabular}
\end{center}
\end{table}

\subsection{The size of $|\mathfrak{F}(N)|$ when $N$ is composite}
Generally, for a given integer $N$, the numbers $k_i$ satisfy the following inequality 
\begin{displaymath}
 \sum_{\stackrel {i<N,} {gcd(i,N)=1}} k_i \leq \phi(N) - 2 - 2*A(2,l(N)) -(2^u-2)
\end{displaymath}

Here, $u$ is the number of distinct prime factors of $N$, the function 
$A(2,l(N))$ gives the number of integers in $[2,l(N)]$ that are co-prime to $N$.
Note that when $N$ is prime, $A(2,l(N))$ is zero, and $2^u-2$ is zero.

Like for prime $N$, we have compared the ratio $\frac{\phi(N)-2}{2|\mathfrak{F}(N)|}$ and the quantity $\log \log(N)$. For many composite $N$, $\frac{\phi(N)-2}{2|\mathfrak{F}(N)|}$ is about $5*\log \log(N)$. We are yet to confirm this behaviour with more data.

\section{Procedure to compute pairs $(r,f)$ with $r>N$}
\noindent Throughout the paper we are interested in computing pairs $(r,f)$ such that
$[\mathcal{X}(N,r,f)] > 0$ and $r<N$. We now ask whether
there exists an $r>N$ such that $[\mathcal{X}(N,r,f)] > 0$.
The answer is positive. Indeed, there exist infinitely many such $r>N$. We prove this result using 
Dirichlet's theorem on primes in arithmetic progressions.
\begin{itemize}
 \item Let $a\geq 1$ be an integer co-prime to $N$ 
 \item Let $b$ be the integer such that $ab \equiv 1\pmod{N+a}$. Such a number $b$
  exists since $a$ and $N+a$ are co-prime
 \end{itemize} 
 We have the following result.
 \begin{lemma}
  If $r = b+i(N+a)$ is prime for some $i\geq 2$, then 
  $\mathcal{X}(N,r,N+a)$ consists of all positive values. 
 \end{lemma}
\proof{Since $r>N$, $l(Nr)=N$. Clearly, $N+a$ divides $Nr+1$.  
From the following computation, $N+a<\sqrt{Nr+1}$. 
\begin{eqnarray}
 Nr+1 - (N+a)^2 \geq N^2 + Nb+1-a^2 > 0
\end{eqnarray}
Thus, $N+a$ is a divisor of $Nr+1$ with $l(Nr)<N+a<\sqrt{Nr+1}$. Therefore, 
 $\mathcal{X}(N,r,N+a)$ consists of all positive values. } \hfill $\Box$ \\

 \noindent \textbf{Dirichlet's Theorem.} For co-prime integers $s$ and $t$, the arithmetic progression $A(s,t)$  
 consists of infinitely many primes. \\
 
 Since $b$ and $N+a$ are known to be co-prime, the progression $A(b,N+a)$ consists of infinitely many primes. 
 This proves the result that \textit{there exist infinitely many $r>N$ such that 
  $[\mathcal{X}(N,r,f)] > 0$}.
  
 \subsubsection{Complexity of finding such an $r$}
 Since the the above method requires a prime $r$,  it amounts to finding a prime in an arithmetic 
 progression efficiently.  But, the following result shows that finding such an $r$ may not be 
 efficient as of now. \\
 
 \noindent \textbf{On the least prime in an arithmetic progression:}  Linnik result proves that there exist 
 some constants $c,L$ such that least prime in the progression
 is less than $cd^{L}$. According to the result of Heath Brown, $c = 2.5$, $L = 5.2$. 
 Thus, the 
 least prime in $r+i(N+a)$ is less than $2.5(N+a)^5.2 < 92*N^{5.2}$. A brute force  
 method requires $O(N^6)$ primality tests to find the least prime. Even with catch 
 that half the numbers in the progressions are divisible by 2, 1/3 of numbers are 
 divisible by 3, and so on, the number of primality tests required to be 
 still in polynomial in $N$, which is exponential in input size, i.e., 
 $\log(N)$.

\end{document}